\documentclass[10pt]{article}
\pdfoutput=1
\usepackage{amssymb}
\usepackage{amsfonts}
\usepackage{amsmath}
\usepackage{color}
\usepackage[pdftex]{graphicx}
\usepackage{a4wide}
\usepackage[T1]{fontenc}
\usepackage[latin1]{inputenc}
\usepackage{eucal,mathrsfs,enumerate,ifthen,boxedminipage,fancybox}
\newcommand{\lapB}{\Delta_\sfg\,}
\newcommand{\tildina}{~}
\newcommand{\R}{\mathbb{R}}

\newcommand{\M}{\mathbb{M}}

\newcommand{\FF}{\mathscr{E}}
\newcommand{\mM}{\mathrm{X}}
\newcommand{\Mm}{\mathrm{X}}
\newcommand{\Xm}{\mathrm{X}}
\newcommand{\FX}{F}
\newcommand{\Fm}{F}
\newcommand{\Mn}{\mathrm{M}^n}

\newcommand{\CC}{\mathscr{C}}
\newcommand{\PP}{\mathscr{P}}

\newcommand{\sfS}{\mathsf{S}}
\newcommand{\sfs}{\mathsf{s}}
\newcommand{\sfE}{\mathsf{E}}
\renewcommand{\SS}{\mathcal{S}}
\newcommand{\sfu}{\gamma}
\newcommand{\sfg}{\mathsf{g}}
\newcommand{\pr}{U}
\newcommand{\U}{e}
\newcommand{\restr}[1]{\lower3pt\hbox{$|_{#1}$}}
\newcommand{\Restr}[1]{\Big|_{#1}}
\newcommand{\ggamma}{{\mbox{\boldmath$\gamma$}}}
\newcommand{\ssigma}{{\mbox{\boldmath$\sigma$}}}

\newcommand{\eps}{\varepsilon}

\newcommand{\rhozero}{\rho_0}

\newcommand{\la}{\langle}
\newcommand{\ra}{\rangle}

\newtheorem{teo}{Theorem}[section]
\newtheorem{lemma}[teo]{Lemma}
\newtheorem{remark}[teo]{Remark}
\newtheorem{corollario}[teo]{Corollary}
\newtheorem{prop}[teo]{Proposition}

\newcommand{\topref}[2]{\stackrel{\eqref{#1}}#2}

\numberwithin{equation}{section}

\newenvironment{proof}{\removelastskip\par\medskip
\noindent{\em Proof.} \rm}{\penalty-20\null\hfill$\square$\par\medbreak}

\newcommand{\freccia}{\rightarrow}

\newcommand{\Vol}{{\rm V}}
\renewcommand{\d}{{\rm d}}
\newcommand{\Dist}{\mathsf{d}}
\begin{document}

\title{Eulerian calculus for the displacement convexity in the
  Wasserstein distance}   
\author{Sara Daneri\\
  S.I.S.S.A., Trieste
  \footnote{S.I.S.S.A., Via Beirut 2-4, 34014, Trieste, Italy.\hfill\break 
  \texttt{e-mail:} \textsf{daneri@sissa.it}.}
  \and 
  Giuseppe Savar\'e\\
  Universit\`a di Pavia
  \footnote{Department of Mathematics, Via
    Ferrata 1, 27100, Pavia, Italy.\hfill\break
    \texttt{e-mail:} \textsf{giuseppe.savare@unipv.it},\hfill\break
  \texttt{web:} \textsf{http://www.imati.cnr.it/$\sim$savare/}}}
\date{January 15, 2008}    
\maketitle
\begin{abstract}
  In this paper we give a new proof
  of the (strong) displacement convexity of a class of integral functionals
  defined on a compact Riemannian manifold satisfying a lower
  Ricci curvature bound.
  Our approach does not rely on existence and regularity
  results for optimal transport maps on Riemannian manifolds, but
  it is based on the Eulerian point of view recently introduced by
  \textsc{Otto-Westdickenberg} in \cite{OttoW} and
  on the metric characterization of the gradient flows generated by
  the functionals in the Wasserstein space.
  \end{abstract}
\section{Introduction}

In this paper we give a new proof, based on a gradient flow approach
and on the Eulerian point of view
introduced by \cite{OttoW},
of the so called ``displacement convexity''
for integral functionals as
\begin{equation}\label{entropintro}
  \FF(\mu):=\int\nolimits_{\M}\U(\rho)\,\d\Vol +\U'(\infty)\,\mu^\perp(\M), \quad
  \rho=\frac{\d\mu}{\d\Vol },
\end{equation}
where $\mu$ is a Borel probability measure %
on a compact, connected
Riemannian manifold without boundary $(\M,\sfg)$, $\Vol$ is the volume
measure on $\M$ induced by the metric tensor $\sfg$,
$\mu^\perp$ is the singular part of $\mu$ with respect to $\Vol$,
$\U:[0,+\infty)\to \R$ is a smooth convex function
satisfying the so called McCann conditions (see \eqref{1} below),
and $\U'(\infty)=\lim\limits_{r\to +\infty}\frac{\U(r)}{r}$.
When $\U$ has a superlinear growth, $\U'(\infty)=+\infty$ so that
$\mu$ should be absolutely continuous with respect to $\Vol$ when
$\FF(\mu)$ is finite.
\paragraph{Displacement convexity for integral functionals.}
The notion of \emph{displacement convexity} has been introduced by
\textsc{McCann} \cite{MC2} to study the behavior of integral
functionals like \eqref{entropintro} along optimal transportation
paths, i.e.\ geodesics in the space of Borel probability measures
$\PP(\M)$  endowed with the $L^2$-Kantorovich-Rubinstein-Wasserstein distance.

Recall that (the square of) this distance can be defined by the following optimal transport problem
\begin{equation}\label{Wass}
  \begin{aligned}
    W^2_2(\mu^0,\mu^1)&:=\min\Big\{
    \int\nolimits_{\M\times\M}\Dist^2(x,y)\,\d\ssigma(x,y):
    \ssigma\in\PP(\M\times\M),\\&
    \ssigma(\M\times B)=\mu^0(B),\
    \ssigma(B\times \M)=\mu^1(B)\quad
    \forall\,B\text{ Borel set in }\M\Big\},
  \end{aligned}
\end{equation}
for the cost function induced by the Riemannian distance $\Dist$
on the manifold $\M$.
We keep the usual notation to
denote by $\PP_2(\M)$ the metric space $(\PP(\M),W_2)$, that is called Wasserstein space; being
$\M$ compact, $W_2$ induces the topology of the weak convergence of
probability measures (i.e., the weak$^*$ topology associated
to the duality of $\PP(\M)$ with $C^0(\M)$).

As in any metric space, (minimal, constant speed) \emph{geodesics} can be
defined
as curves $\mu:s\in [0,1]\mapsto \mu^s\in \PP_2(\M)$
between $\mu^0$ and $\mu^1$ satisfying
\begin{equation}\label{geod}
W_2(\mu^r,\mu^s)=|s-r|\,W_2(\mu^0,\mu^1)\quad\forall\:\,0\leq r\leq s\leq1.
\end{equation}
A functional $\FF:\PP(\M)\freccia(-\infty,+\infty]$ is
then \emph{(strongly) displacement convex} (or, more generally,
\emph{displacement $\lambda$-convex} for some $\lambda\in \R$)
if, \emph{for all} Wasserstein geodesics $\{\mu^s\}_{0\leq s\leq
  1}\subset\PP_2(\M)$, we have
\begin{equation}\label{convexl}
  \FF(\mu^s)\leq(1-s)\FF(\mu^0)+s\FF(\mu^1)-\frac{\lambda}{2} s(1-s)
  W_2^2(\mu^0,\mu^1),\quad\forall\,s\in[0,1].
\end{equation}
A weaker notion is also often considered: one can ask that there
exists \emph{at least one} geodesic connecting $\mu^0$ to $\mu^1$ along which
\eqref{convexl} holds.

The term ``displacement convexity'' arises from the
strictly related concept of ``displacement interpolation'' introduced
by \cite{MC2} in the Euclidean case $\M=\R^d$; in a general metric
setting, property \eqref{convexl} is simply called, as in the
Riemannian case, ``$\lambda-$geodesic convexity'' (or ``geodesic
convexity'' if $\lambda=0$).

\noindent It is possible to show \cite{Ambrosio-Gigli-Savare05} that
the measures $\mu^s$ can also be defined through the formula
\begin{equation}
  \label{new:1}
  \mu^s(B):=\ssigma\big(\{(x,y)\in \R^d\times\R^d: (1-s)x+s y\in B\}\big),\quad
  \text{where }\ssigma\text{ is a minimizer of
  \eqref{Wass}.}
\end{equation}
A similar construction can also be performed in a Riemannian manifold
\cite{Lott-Villani07,Sturm06I,Lisini07}:
the segments $s\mapsto (1-s)x+sy$ should be substituted
by a Borel 
map $\ggamma:\M\times\M\to C^0([0,1];\M)$
that at each couple $(x,y)\in\M\times\M$ associate
a (minimal, constant speed) geodesic $s\mapsto \ggamma^s(x,y)$ in $\M$
connecting $x$ to $y$.
We have the representation formula
\begin{equation}
  \label{new:2}
   \mu^s(B):=\ssigma\big(\{(x,y)\in \M\times\M: \ggamma^s(x,y)\in B\}\big),\quad
  \text{where }\ssigma\text{ is a minimizer of
  \eqref{Wass}.}
\end{equation}
After the pioneering paper \cite{MC2}, the notion of displacement
convexity for integral functionals found applications in many
different fields, as Functional inequalities
\cite{Otto-Villani00,Agueh-Ghoussoub-Kang04,CMS1}, generation, contraction, and
asymptotic properties of diffusion equations and Gradient flows
\cite{Otto,Agueh05,OttoW,Ambrosio-Gigli-Savare05,Carrillo-McCann-Villani06,Ambrosio-Savare06},
Riemannian Geometry and synthetic study of Metric-Measure spaces
\cite{Sturm06I,Lott-Villani07}.

In the context of Riemannian manifolds it turns out that
displacement $\lambda$-convexity of certain classes of entropy
functionals is equivalent to a lower bound for the Ricci curvature of
the manifold. The connection between displacement convexity and Ricci
curvature, introduced by \cite{Otto-Villani00}, was then further deeply
studied by \cite{Otto-Villani00,CMS1,CMS2,Sturm06I}; the equivalence
has been proved by Sturm and Von Renesse in \cite{SturmVonR}, who
considered the case in which the domain of the functional consists
only of measures that are absolutely continuous with respect to the volume
measure, and then completed by Lott and Villani \cite{Lott-Villani07} (with
the remarks made in \cite{FigVill}, where convexity in the strong form
has been proved), who extended the previous results
to the functionals defined by \eqref{entropintro} on all $\PP(\M)$.
We refer to the forthcoming monograph \cite{VillaniOldNew} for further
references, details, and discussions.

The strategy followed by the authors of \cite{CMS1} (and by all the
following contributions) in order to
characterize the displacement convexity of entropy functionals relies
on a characterization of optimal transportation and Wasserstein
geodesics \cite{MC1} and on a
careful study of the Jacobian properties of the exponential function which
are crucial to estimate the integral functionals along this class of curves.
The lack of regularity of Wasserstein geodesics and the 
lack of global smoothness of the squared distance function $\Dist^2$ on the
manifold $\M$ (due to the existence of the cut-locus) require a
careful use of non-smooth analysis arguments and non trivial approximation processes to extend the results to geodesics between arbitrary measures (see \cite{Lott-Villani07, FigVill}).

The main result is the following
\begin{teo}\label{teo0}
  $\mathrm{(I)}$ If $\U\in   C^\infty(0,+\infty)$ satisfies the \textsc{McCann}
  conditions:
  \begin{equation}
    \pr(\rho):=\rho \U'(\rho)-\big(\U(\rho)-\U(0_+)\big)\geq0,
    \qquad
    \rho \pr'(\rho)-\biggl(1-\frac{1}{n}\biggr)\pr(\rho)\geq0,\quad
    n:=\dim(\M)>1
    \label{1}
  \end{equation}
  and $\M$ has \emph{nonnegative} Ricci curvature, 
  then the functional $\FF$ defined by \eqref{entropintro}
  is \emph{(strongly) displacement convex.}
  
  $\mathrm{(II)}$ If $\FF$ is the relative entropy functional,
  corresponding to $\U(\rho)=\rho \log\rho$ (which satisfies
  \eqref{1} in any dimension) in \eqref{entropintro},
  and there exists $\lambda\in\R$ such that 
  \begin{equation}\label{riclambda}
    \mathrm{Ric}_x\,(\xi,\xi)\geq\lambda \la \xi,\xi\ra_{\sfg_x}
    \quad\forall\,x\in\M,\quad\forall\,\xi\in T_x\M, 
  \end{equation}
  then the functional $\FF$ defined by \eqref{entropintro}
  is \emph{(strongly) displacement $\lambda$-convex.}
\end{teo}
\begin{remark}
  \upshape
  \label{rem:examples}
  Besides the logarithmic entropy corresponding to $\U(\rho)=\rho\log
  \rho$ (and $\pr(\rho)=\rho$), typical examples of functionals that satisfy properties
  \eqref{1} are
  \begin{equation}\label{ex}
    \U(\rho)=\tfrac1{m-1}\rho^m,\quad
    \pr(\rho)=\rho^m,\quad
    m\geq1-\tfrac{1}{n}.
  \end{equation}
  We recall that assumptions \eqref{1} imply the
  convexity of the function $\rho\mapsto \U(\rho)$ (since the dimension
  $n$ is greater than $1$, they are in fact more restrictive).
\end{remark}
\paragraph{Aim of the paper: an Eulerian approach to displacement convexity.}
In this paper we present an alternative proof of Theorem \ref{teo0},
which does not rely on the existence and smoothness of optimal
transport maps and geodesics for the Wasserstein distance.

Our strategy can be described in three steps:
\begin{enumerate}
  \item
    Following the approach suggested by \textsc{Otto-Westdickenberg}
    in \cite{OttoW}, we work in the subspace $\PP_2^{ar}(\M)$ 
    of measures with
    smooth and positive densities and we
    use the ``Riemannian'' formula
    for the Wasserstein distance, originally introduced in the Euclidean
    framework by \textsc{Benamou-Brenier} \cite{Benamou-Brenier00}:
    if $\mu^i=\rho^i\,\Vol\in \PP_2^{ar}(\M)$,
    $i=0,1$, then \cite[Prop. 4.3]{OttoW}
  \begin{equation}\label{riemann0}
    W_2^2(\mu^0,\mu^1)=\underset{\CC(\mu^0,\mu^1)}{\inf}\Bigl
    \{\int_0^1\int_{M}|\nabla\phi^s|^2\rho^s\,\d\Vol \,\d
    s\Bigr\}\qquad
    \forall\, \mu^0,\mu^1\in \PP_2^{ar}(\M)
  \end{equation}
  where
  \begin{equation}
  \begin{aligned}  
    \CC(\mu^0,\mu^1)=\Big\{&(\rho,\phi): \rho\in
    C^{\infty}([0,1]\times\M;\R_+),\quad
    \phi\in C^{\infty}([0,1]\times\M)\\ 
    &\partial_s\rho^s+\nabla\cdot(\rho^s\nabla\phi^s)=0\, \text{ in
    }(0,1)\times\M,\quad \mu^i=\rho^i\,\Vol\Big\}.
  \end{aligned}
  \label{eq:DS:32}
  \end{equation}
  Even though the Wasserstein space can't be endowed with a
  smooth Riemannian structure, \eqref{eq:DS:32}
  still shows a ``Riemannian'' characterization of the
  Wasserstein distance on $\PP_2^{ar}(\M)$.
\item
  The second important fact, originally showed by
  the so-called ''Otto calculus'' in 
  \cite{Otto}, is
  that
  the nonlinear diffusion equation 
  \begin{equation}\label{eq0}
    \partial_t\rho_t-\lapB \pr(\rho_t)=0 \quad\text{ in
    }[0,+\infty)\times\M,
    \qquad
    \rho\restr{t=0}=\rho_0,
  \end{equation}
  where $U:\R^+\freccia\R$ is the function defined in \eqref{1} and
  $\lapB$ is the Laplace-Beltrami operator on $\M$,
  is the gradient flow of the functional \eqref{entropintro}
  in $\PP_2(\M)$. Indeed, \eqref{eq0} corresponds to the heat equation if $U$ is the logarithmic entropy and to the porous medium equation if $U$ is defined by \eqref{ex}.
  
  Starting directly from \eqref{riemann0} and owing to the fact
  that the flow generated by \eqref{eq0}
  preserves smooth and positive densities, when
  $\mathrm{Ric}(\M)\ge0$ we shall show that
  the measures $\mu_t=\tildina\rho_t\Vol\in\tildina\PP_2^{ar}(\M)$
  associated to the solutions of
  \eqref{eq0} also solve
  the Evolution Variational
  Inequality (E.V.I.) 
  \begin{align}\label{flussograd0}
    \frac{1}{2}\frac{\d^+}{\d t}
    W_2^2(\nu,\mu_t)\leq\FF(\nu)-\FF(\mu_t)\quad
    \forall\,t\geq0,\,\nu\in \PP_2^{ar}(\M),
  \end{align}
  which has been introduced in \cite{Ambrosio-Gigli-Savare05}
  as a purely metric characterization of the gradient flows
  of geodesically convex functionals in metric spaces
  (and in particular
  in $\PP_2(\R^d)$); here 
  \begin{equation}
    \label{rightd}
    \frac{\d ^+}{\d t}\zeta(t)=
    \limsup_{h\downarrow0}\frac{\zeta(t+h)-\zeta(t)}h
  \end{equation}
  for every real function $\zeta:[0,+\infty)\to\R$.
  
  When $\mathrm{Ric}\,(\M)\geq\lambda$ (a shorthand for
  \eqref{riclambda}), we also show that
  the solutions of the heat equation satisfy the
  modified inequality
  \begin{align}\label{flussograd0l}
    \frac{1}{2}\frac{\d^+}{\d t}W_2^2(\nu,\mu_t)+\frac{\lambda}{2}W_2^2(\nu,\mu_t)\leq\FF(\nu)-\FF(\mu_t)\quad
    \forall\,t\geq0,\,\nu\in \PP_2^{ar}(\M),
  \end{align}
  where $\FF$ is the relative entropy functional whose integrand function is
  $\U(\rho)=\rho\log\rho$. Note that \eqref{flussograd0l} reduces to
  \eqref{flussograd0} when $\lambda=0$.
  In order to prove \eqref{flussograd0} and
  \eqref{flussograd0l}, we propose an ``Eulerian'' strategy
  which could be adapted to more general situations.
\item
  The third crucial fact is the following: whenever a functional
  $\FF$ satisfies \eqref{flussograd0}
  (or, more generally, \eqref{flussograd0l})
  for a given semigroup $\SS_t:\mu_0=\rho_0\Vol\mapsto
  \mu_t=\rho_t\Vol$ in $\PP_2^{ar}(\M)$,
  $\FF$ is displacement convex (resp.\ displacement $\lambda$-convex).
  Thus the question of the behavior of $\FF$ along geodesics can
  be reduced to a differential estimate of $\FF$ along the smooth and
  positive solutions of 
  its gradient flow.
\end{enumerate}
\paragraph{Plan of the paper.}
In Section 2 we present the main ideas of our approach in the simplified
(finite-dimensional and smooth) setting of geodesically convex
functions on Riemannian manifolds.
We think that these ideas are sufficiently general to be useful in
other circumstances, at least for distances
which admits a Riemannian characterization as \eqref{riemann0},
see e.g.\ \cite{Dolbeault-Nazaret-Savare-toapp,Carrillo-Lisini-Savare-inprep}

After a brief review of the definition of (gradient) $\lambda$-flows
in arbitrary metric spaces (basically following the
ideas of \cite{Ambrosio-Gigli-Savare05}),
we present in Section 3 our first result, showing that the existence
of a flow satisfying the E.V.I.\ \eqref{flussograd0l}
(even on a dense subset of initial data, such as $\PP_2^{ar}(\M)$)
entails the (strong) displacement $\lambda$-convexity of the functional $\FF$.

Following the strategy explained in the second section, in the
last two sections 
we prove the differential estimates showing that \eqref{eq0} satisfies
\eqref{flussograd0} (in Section 4) or, in the case
of the Heat equation,
\eqref{flussograd0l} (in Section 5).
\section{Gradient flows and geodesic convexity in a smooth setting}
\paragraph{Contraction semigroups and action integrals.}
In order to explain the main point of our strategy, let us first
consider the simple setting of a smooth function
$\Fm:\Mm\freccia\R$ on a complete Riemannian manifold $\Mm$ with metric
$\la \cdot,\cdot\ra_g$, (squared) norm $|\xi|_g^2=
\la \xi,\xi\ra_g$, and the endowed
Riemannian distance
\begin{equation}
  \label{eq:DS:3}
  d^2(u,v):=\min\Big\{\int\nolimits_0^1
  \big|\dot{\gamma}^s|_g^2\,\d s,
  \quad\gamma:[0,1]\freccia\Mm,\:\gamma^0=v,\ \gamma^1=u\Big\}.
\end{equation}
In a smooth setting, the geodesic $\lambda$-convexity of $\Fm$ can be
expressed through the differential condition
\begin{equation}
  \label{eq:DS:2}
  \frac {\d^2}{\d s^2}\Fm(\gamma^s)\ge \lambda\, |\dot\gamma^s|_g^2
\end{equation}
along any geodesic curve $\gamma$ minimizing
\eqref{eq:DS:3}.
As we discussed in the introduction, the direct computation of
\eqref{eq:DS:2}
could be difficult in a non-smooth, infinite dimensional setting; it
is therefore important to find equivalent conditions which avoid
twofold differentiation along geodesics.
One possibility, suggested in \cite{OttoW}, is to
find equivalent conditions to geodesic $\lambda$-convexity in terms of
the gradient flow generated by $\Fm$.

Let us recall that the gradient flow of $\Fm$
is a continuous semigroup of (time-dependent) maps $\sfS_t:\Mm
\to \Mm$, $t\in [0,+\infty)$,
which at every initial datum $u$ associate
the curve $u_t:=\sfS_t(u)$ solution of the differential equation
 \begin{equation}\label{flusso}
   \dot u_t=-\nabla \Fm(u_t)\quad\forall\,t\ge 0,\quad
   u_0=u.
\end{equation}
It is well known that, when $\Fm$ is geodesically $\lambda$-convex,
$\sfS_t$ is $\lambda$-contracting, i.e.\
\begin{equation}
  \label{eq:DS:1}
  d^2(\sfS_t(u),\sfS_t(v))\le e^{-2\lambda t}d^2(u,v)
  \qquad \forall\, u,v\in \Mm.
\end{equation}
By the semigroup property, \eqref{eq:DS:1} is also equivalent to
the differential inequality (see \eqref{rightd})
\begin{equation}
  \label{eq:DS:6}
  \frac {\d^+}{\d t}d^2(\sfS_t(u),\sfS_t(v))\Big|_{t=0}\le
  -2\lambda\, d^2(u,v)\quad
  \forall\, u,v\in \Mm.
\end{equation}
\cite{OttoW} reverts this argument and 
observes that it could be easier to directly prove
\eqref{eq:DS:6} by a
differential estimate
involving only the action of the semigroup along smooth curves;
as a byproduct, one should obtain the convexity of $\Fm$.
To this aim, they consider a smooth curve $\gamma^s$, $s\in [0,1]$,
connecting $v$ to $u$, 
and the action integral $\mathscr A_t$ associated
to its  smooth perturbation
\begin{equation}
  \label{eq:DS:4}
  \sfu^s_t:=\sfS_t(\gamma^s),\quad
  A^s_t:=  \big|\partial_s \sfu^s_t\big|_g^2,\quad
  \mathscr A_t:=\int\nolimits_0^1 A^s_t
  \,\d s,
\end{equation}
where $\partial_s \gamma,\partial_t \gamma$ denotes the tangent vectors
in $T_\gamma\Mm$ obtained by
differentiating w.r.t.\ $s$ and $t$ respectively.
Since, by the very definition of $d$,
\begin{equation}
  \label{eq:DS:7}
  d^2(\sfS_t(v),\sfS_t(u))\le \mathscr A_t
\end{equation}
and for every $\eps>0$
one can always find a curve $\gamma^s$ so that $\mathscr A_0\le
d^2(u,v)+\eps$
(in a smooth setting one can take $\eps=0$),
\eqref{eq:DS:6} surely holds if one can prove that
\begin{equation}
  \label{eq:DS:8}
  \frac {\d^+}{\d t}\mathscr A_t\Big|_{t=0}\le -2\lambda\,\mathscr A_0,
  \quad\text{or its pointwise version}\quad
  \frac {\partial}{\partial t}\Big|_{t=0}A^s_t
  \le -2\lambda\,A^s_0.
\end{equation}
Having obtained the contraction property from
\eqref{eq:DS:8}, it still remains open how to deduce
that $\Fm$ is geodesically convex.
Notice that along an arbitrary curve $\eta^s$
\begin{equation}
  \label{eq:DS:9}
  \frac{\partial}{\partial s}
  \Fm(\eta^s)=\la\nabla \Fm(\eta^s),\partial_s \eta^s\ra_g=
  -\la\partial_r\sfS_r(\eta^s)\restr{r=0},\partial_s \eta^s\ra_g;
\end{equation}
applied to $\eta^s:=\sfu^s_t$, \eqref{eq:DS:9} and the semigroup property
$\sfS_r(\sfu^s_t)=\sfu^s_{t+r}$ yield
\begin{equation}
  \label{eq:DS:9bis}
  \frac{\partial}{\partial
    s}
  \Fm(\sfu^s_t)=
  -\la\partial_t \sfu^s_t,\partial_s \sfu^s_t\ra_g.
\end{equation}
\emph{In a smooth setting} we can assume that $\gamma^s$ is
a minimal geodesic; operating a further differentiation with
respect to $s$, we obtain
\begin{align}
  \label{eq:DS:10}
  \frac{\partial^2}{\partial
    s^2}
  \Fm(\sfu^s)&\topref{eq:DS:9}=
  -\frac{\partial}{\partial
    s}\la\partial_t \sfu^s_t,\partial_s \sfu^s_t\ra_g
  \Restr{t=0}=
  -\la D_{\partial_s}\partial_t \sfu^s_t,\partial_s \sfu^s_t\ra_g
  -\la\partial_t \sfu^s_t,D_{\partial_s}\partial_s \sfu^s_t\ra_g\Big|_{t=0}
  \\&=
  -\la D_{\partial_s}\partial_t \sfu^s_t,\partial_s \sfu^s_t\ra_g\Big|_{t=0}
  =
  -\la D_{\partial_t}\partial_s \sfu^s_t,\partial_s \sfu^s_t\ra_g\Big|_{t=0}=
  -\frac 12 \frac \partial{\partial t}
  \la\partial_s \sfu^s_t,\partial_s \sfu^s_t\ra_g\Big|_{t=0}
  \notag\\&\label{eq:DS:11}
  \topref{eq:DS:4} =-\frac 12 \frac \partial{\partial t}\Big|_{t=0} A^s_t
  \topref{eq:DS:8}\ge
  \lambda\,
  \big|\partial_s \sfu^s\big|_g^2,
\end{align}
where we used the standard properties of the
covariant differentiations $D_{\partial_s}, D_{\partial_t}$ 
and, in \eqref{eq:DS:10},
the fact that at $t=0$ $D_{\partial_s}\partial_s \sfu^s_t=0$,
being $\sfu^s_t=\gamma^s$ a geodesic.
\paragraph{A metric derivation of convexity.}
Even if the previous differential argument shows that \eqref{eq:DS:8}
implies geodesic $\lambda$-convexity, it still requires nice smooth
properties on geodesics and covariant differentiation, which could be hard
to extend to a non smooth setting.

This is not at all surprising, since the contraction property
\eqref{eq:DS:6}
and its action-differential characterization \eqref{eq:DS:8} do not
carry all the information linking the semigroup $\sfS$ to $\Fm$:
in order to conclude the argument
in \eqref{eq:DS:10} we had therefore to insert the information coming from
\eqref{eq:DS:9}.

To overcome these difficulties, we shall deal with a
more precise metric characterization of $\sfS$ than \eqref{eq:DS:1}.
As it has been proposed and studied in \cite{Ambrosio-Gigli-Savare05},
gradient flows of geodesically $\lambda$-convex functionals
in ``almost'' Euclidean settings should satisfy a purely metric formulation
in terms of the Evolution Variational Inequality
\begin{equation}\label{evi0}
  \frac{1}{2}\frac{\d^+}{\d t}d^2(\sfS_t( u),v)+
  \frac \lambda 2d^2(\sfS_t( u),v)+\Fm(\sfS_t(u))
  \leq \Fm(v), \quad\forall\,v\in\mM,\, t>0.
\end{equation} 
It can be proved (see \cite{Ambrosio-Savare06}) that 
\eqref{evi0} characterizes $\sfS$ and implies the contractivity
property \eqref{eq:DS:1}.

As we discussed before, here we invert the usual procedure
(starting from a convex functional, construct its gradient flow)
and we suppose that there exists a smooth flow $\sfS_t$ satisfying
\eqref{evi0}. The following result, whose proof will be postponed
(in a more general form) to Theorem \ref{convex_1} in the next Section,
shows that $\Fm$ is geodesically $\lambda$-convex.
\begin{teo}
  \label{convex_0}
  Suppose that there exists a continuous semigroup of maps 
  $\sfS_t\in C^0(\Mm;\Mm)$, $t\ge0$, satisfying \eqref{evi0}.
  Then for every (minimal, constant speed) geodesic $\gamma:[0,1]\to\Mm$
  \begin{equation}
    \label{eq:DS:12}
    \Fm(\gamma^s)\le (1-s)\Fm(\gamma^0)+s \Fm(\gamma^1)-
    \frac \lambda2s(1-s)d^2(\gamma^0,\gamma^1),\quad\forall\,s\in[0,1]
  \end{equation}
  i.e.\ $\Fm$ is (strongly) geodesically $\lambda$-convex. 
\end{teo}

\paragraph{E.V.I. through action-differential estimates.}
Thanks to Theorem \ref{convex_0},
it is possible to prove the geodesic $\lambda$-convexity of $\Fm$ by
exhibiting a flow $\sfS$ satisfying the E.V.I.\ \eqref{evi0}.
According to the general strategy suggested by
\cite{OttoW},
we want to reduce \eqref{evi0} to a suitable family
of differential inequalities satisfied by the action $A^s_t$
of \eqref{eq:DS:4}.

The idea here is to consider a different family of perturbations
of a given smooth curve $\gamma:[0,1]\to\Mm$, still
induced by the semigroup $\sfS$.
In fact, differently from the contraction estimate \eqref{eq:DS:6}
where we are flowing both the points $u,v$ through $\sfS_t$,
in \eqref{evi0} we want to keep the point $v:=\gamma^0$ fixed and to vary
only $u:=\gamma^1$.
If $\gamma^s$ is a smooth curve connecting them,
it is then natural to consider the new families (see
Figure 1)
\begin{equation}
  \label{eq:DS:19}
  \tilde \sfu^s_t:=\sfS_{st}(\gamma^s)=\sfu^s_{st},
  \quad
  \tilde \Fm^s_t:=\Fm(\tilde\gamma^s_t)\qquad
  s\in [0,1],\ t\ge0.
\end{equation}
\begin{center}
  
  
  



  

  

  
  \includegraphics{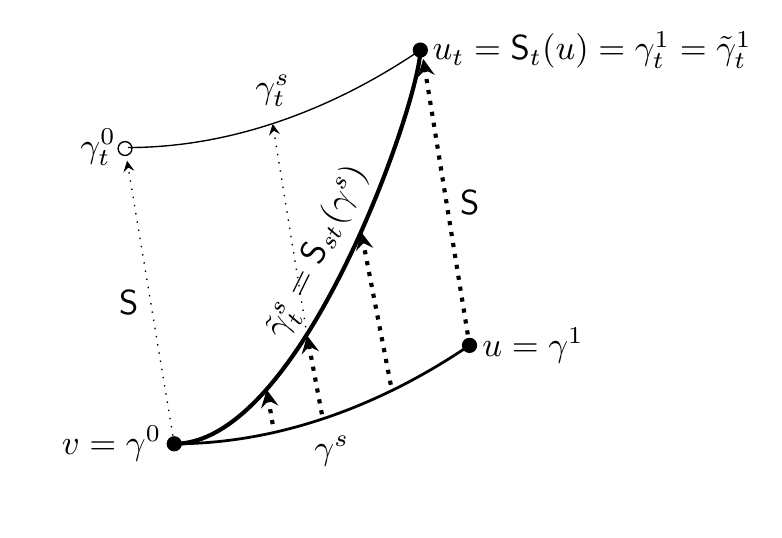}
  \smallbreak
  Figure 1: variation of the curve $\gamma^s$ under the action
  of the semigroup $\sfS$.
\end{center}
Notice that $\tilde \gamma^s_0=\gamma^s$, $\tilde \gamma^0_t=\gamma^0=v$,
$\tilde \gamma^1_t=\sfS_t(\gamma^1)=\sfS_t(u)$.
As before, we introduce the quantities
\begin{equation}
  \label{eq:DS:20}
  \tilde A^s_t:=\big|\partial_s \tilde \sfu^s_t\big|_g^2,
  \quad
  \tilde{\mathscr A}_t:=\int\nolimits_0^1 \tilde A^s_t\,\d s.
\end{equation}
\begin{teo}[A differential inequality linking action and flow]
  \label{thm:main_smooth}
  Suppose that for every smooth curve $\gamma:[0,1]\to \Mm$
  the quantities $\tilde A^s_t,\tilde \Fm^s_t$ induced
  by the flow $\sfS$ through \eqref{eq:DS:19},\eqref{eq:DS:20} satisfy
  \begin{equation}
    \label{eq:DS:26}
    \frac 12\frac\partial{\partial t}\tilde A^{s}_t+
    \frac{\partial}{\partial s}\tilde \Fm^s_t
    \le -\lambda\, s\,
    \tilde A^s_t,\quad\forall\,t\geq0.
  \end{equation}
  Then $\sfS$ satisfies \eqref{evi0},
  it is the gradient flow of $\Fm$, and $\Fm$ is geodesically 
  $\lambda$-convex.
  Moreover, it is sufficient to check \eqref{eq:DS:26} at $t=0$.
\end{teo}
\begin{proof}
  Let us first observe that \eqref{eq:DS:26} yields, after
  an integration with respect to $s$ in $[0,1]$,
  \begin{equation}
    \label{eq:DS:24}
    \frac 12\frac \d{\d t}\tilde{\mathscr A}_t+\tilde \Fm^1_t-
    \tilde \Fm^0_t\le
    -\lambda \int\nolimits_0^1 s\tilde A^s_t\,\d s.
  \end{equation}
  By the semigroup property, it is sufficient to
  prove \eqref{evi0} at $t=0$.
  We choose a geodesic $\gamma^s$ connecting $v$ to $u$ and
  we consider the curves given by
  \eqref{eq:DS:19}. Since
  \begin{equation}
    \label{eq:DS:27}
    d^2(v,\sfS_t(u))\le \int_0^1 \tilde A^s_t\,\d s=
    \tilde{\mathscr A}_t,\quad
    d^2(v,u)=\int_0^1 \tilde A^s_0\,\d s=\tilde{\mathscr A}_0,
    \quad
    \tilde \Fm^1_t=\Fm(\sfS_t(u)),\quad
    \tilde \Fm^0_t=\Fm(v),
  \end{equation}
  by \eqref{eq:DS:24} at $t=0$ we obtain
  \begin{equation}
    \label{eq:DS:28}
    \frac 12 \frac {\d^+}{\d t}d^2(\sfS_t(u),v)\Big|_{t=0}+\Fm(u)-
    \Fm(v)\le
    -\lambda\int\nolimits_0^1 s\, \tilde A^s_0\,\d s=
    -\frac \lambda2 d^2(u,v),
  \end{equation}
  where in the last identity we used the fact that
  $\gamma^s$ is a geodesic and therefore
  $\tilde A^s_0=
  |\partial_s\gamma^s|_g^2$
  is constant in $[0,1]$ and takes the value
  $d^2(\gamma^0,\gamma^1)=d^2(v,u)$.
  
  Since $\tilde\sfu^s_{t_0+t}=\sfS_{st}
  \tilde\sfu^s_{t_0}$ by the semigroup property,
  if $\sfS$
  satisfies \eqref{eq:DS:26} at the initial time $t=0$ for an arbitrary
  smooth curve
  $\gamma$, then it also satisfies \eqref{eq:DS:26} for $t>0$.  
\end{proof}
Our last result provides a simple criterion to check \eqref{eq:DS:26}:
\begin{teo}
  \label{thm:main_reduction}
  Suppose that the flow $\sfS:[0,+\infty)
  \times \Mm\to\Mm$ satisfies \eqref{eq:DS:9}
  for any smooth curve
  $\gamma^s$, let $\gamma^s_t,\tilde\gamma^s_t,A^s_t,\tilde
  A^s_t,\tilde \Fm^s_t$ be
  defined as
  in \eqref{eq:DS:4}, \eqref{eq:DS:19}, and \eqref{eq:DS:20},
  and let us set
  \begin{equation}
    \label{eq:DS:29a}
    \tilde D^s_r:=\frac{1}{2}\lim_{h\downarrow0}h^{-1}\Big(
    \big|\partial_s \gamma^s_{sr+h}\big|_g^2-
    \big|\partial_s \gamma^s_{sr}\big|_g^2\Big),
  \end{equation}
  Then
  \begin{equation}
    \label{eq:DS:21}
    \frac 12\frac\partial{\partial t}\tilde A^{s}_t+
    \frac{\partial}{\partial s}\tilde \Fm^s_t=s \tilde D^s_{t}.
  \end{equation}
  Furthermore, if \eqref{eq:DS:8} holds, then
  \begin{equation}
    \label{eq:DS:22}
    \tilde D^s_t\le -\lambda\, \tilde A^s_t 
  \end{equation}
  and \eqref{eq:DS:26} holds, too, so that $\Fm$ is geodesically 
  $\lambda$-convex,
  and $\sfS$ is its gradient flow.
\end{teo}
\begin{proof}
  Let us set
  \begin{equation}
    \label{eq:DS:30}
    \tilde\gamma^s_{t,\tau}:=\sfS_{\tau}\tilde\gamma^s_t=\gamma^{s}_{st+\tau},
    \quad
    \tilde A^s_{t,\tau}:=\big|\partial_s \tilde\gamma^s_{t,\tau}\big|_g^2,
  \end{equation}
  so that
  \begin{equation}
    \label{eq:DS:31}
    \tilde\gamma^s_{t+h}=\tilde\gamma^s_{t,sh},\quad
    \partial_s \tilde \gamma^s_{t+h}=
    \partial_s \tilde\gamma^s_{t,\tau}+
    h\partial_\tau \tilde\gamma^s_{t,\tau}\Big|_{\tau=sh},\quad
    \tilde D^s_t=\frac12\frac\partial{\partial \tau} \tilde A^s_{t,\tau}\Big|_{\tau=0}
  \end{equation}
  Observe that
  the identity 
 \begin{equation}
   \label{eq:DS:32bis}
   |x+y|_g^2=2\la x+y,y\ra_g+|x|_g^2-|y|_g^2,\qquad
   \forall\, x,y\in T_\gamma\Mn
 \end{equation}
 yields
\begin{align*}
  \tilde A^s_{t+h}&=\big|\partial_s \tilde \gamma^s_{t+h}\big|_g^2
  \topref{eq:DS:31}=
  \big|\partial_s\tilde \gamma^s_{t,\tau}+
  h\partial_\tau \tilde\gamma^s_{t,\tau}\big|_g^2
   \Big|_{\tau=sh}\\&
  \topref{eq:DS:32bis}=
  \Big[2h\la \partial_s\tilde \gamma^s_{t,\tau}+
  h\partial_\tau \tilde\gamma^s_{t,\tau},\partial_\tau\tilde \gamma^s_{t,\tau}\ra+
  \big| \partial_s\tilde \gamma^s_{t,\tau}\big|_g^2-
  h^2 \big|\partial_\tau \tilde\gamma^s_{t,\tau}\big|_g^2\Big]_{\tau=sh}
  \\&=
  2 h\, \la \partial_s \tilde \sfu^s_{t+h},
  \partial_\theta
  \sfS_\theta(\tilde\gamma^s_{t+h}))\ra
  \Big|_{{\theta=0}}+
  \tilde A^s_{t,sh}-o(h)
  \topref{eq:DS:9}=
  -2h\,\frac\partial{\partial s}\Fm(\tilde\gamma^s_{t+h})+
   \tilde A^s_{t,sh}-o(h).
\end{align*}
We thus get
\begin{equation}
  \label{eq:DS:34}
  \frac 1{2h}\big(\tilde A^s_{t+h}-\tilde A^s_t\big)+
  \frac {\partial}{\partial s}\Fm(\tilde\gamma^s_{t+h})=
  \frac 1 {2h}\big(\tilde A^s_{t,sh}-\tilde A^s_t\big) -o(1),
\end{equation}
so that, passing to the limit as $h\downarrow0$ we get
\eqref{eq:DS:21}.
\end{proof}
\begin{remark}
  \label{rem:tedious}
  \upshape
  Notice that the remainder term $o(1)$ in \eqref{eq:DS:34}
  is non-negative, so it can be simply neglected,
  if one is just interested to the inequality
  \eqref{eq:DS:26}.
\end{remark}
\section{Gradient flows and geodesic convexity in a metric setting}
In this section we will briefly recall some basic definitions and
properties
of gradient flows in a metric setting and we will prove
Theorem \ref{convex_0} in a slightly more general framework.

Let $(\Xm,d)$ be a metric space (not necessarily complete)
and let $\FX:\Xm\to (-\infty,+\infty]$ be a lower semicontinuous
functional,
whose proper domain $D(\FX):=\big\{w\in \Xm:\FX(w)<+\infty\big\}$
is dense in $\Xm$ (otherwise we can always restrict
all the next statements to the closure of $D(\FX)$ in $X$).
We also assume that $\FX$ is bounded from below,
i.e.\ $F_{\rm inf}:=\inf_{u\in \Xm}F(u)>-\infty$.

A $C^0$-semigroup $\sfS$ in $C^0(\Xm;\Xm)$ is a family
$\sfS_t$, $t\ge0$, of continuous maps in $\Xm$ such that
\begin{equation}
  \label{eq:DS:35}
  \sfS_{t+h}(u)=\sfS_{h}\big(\sfS_t(u)\big),\quad
  \lim_{t\downarrow0}\sfS_t(u)=\sfS_0(u)=u\quad
  \forall\, u\in \Xm,\ t,h\ge0.
\end{equation}
Given a real number $\lambda\in \R$, we say that
$\sfS$ is the $\lambda$-(gradient) flow of $\FX$ if
it satisfies
\begin{subequations}
  \label{eq:1}
  \begin{gather}
    \label{sdom}\text{$\sfS_t(\Xm)\subset D(\FX)$ for every $t>0$;}\\
    \label{notincr}\text{the map $t\mapsto \FX(\sfS_t(u))$
      is not increasing in $(0,+\infty)$;}\\
    \label{evi1}
    \frac{1}{2}\frac{\d^+}{\d t}d^2(\sfS_t( u),v)+
    \frac \lambda 2d^2(\sfS_t( u),v)+\FX(\sfS_t(u))
    \leq \FX(v), \quad\forall\,u\in\Xm,\,v\in D(\FX),\, t\geq0.
  \end{gather}
\end{subequations}
Clearly, if $\sfS$ is a $\lambda$-flow for $\FX$, then it is also
a $\lambda'$-flow for every $\lambda'\le \lambda$.
The next proposition collects some useful properties of
$\lambda$-flows. 
\begin{prop}[Integral characterization of flows and contraction]
  \label{prop:useful}
  A $C^0$-semigroup $\sfS$ satisfies
  $(\ref{eq:1}a,b,c)$ if and only if
  it satisfies the following
  integrated form
  \begin{equation}
    \label{eq:DS:13tris}
    \frac {e^{\lambda (t_1-t_0)}}2
    d^2(\sfS_{t_1}(u),v)-
    \frac 1
    2d^2(\sfS_{t_0}(u),v)
    \le
    \sfE_\lambda(t_1-t_0)
    \Big(F(v)-F(\sfS_{t_1}(u))\Big)
    \quad
    \forall\,0\le t_0<t_1,
  \end{equation}
  for every $u\in \Xm,\ v\in D(\FX)$,
  where $\sfE_\lambda(t):=
  \int_0^t e^{\lambda r}\,\d r=
  \begin{cases}
    \frac{e^{\lambda t}-1}{\lambda}&\text{if }\lambda\neq0,\\
    t&\text{if }\lambda=0. 
  \end{cases}$\\
  In particular $\sfS$   satisfies the uniform regularization bound
  \begin{equation}
    \label{eq:DS:37}
    \FX(\sfS_t(u))\le \FX(v)+\frac 1{2\,\sfE_\lambda(t)}
    d^2(u,v)
    \quad
    \forall\, u\in \Xm,\ v\in D(\FX),\ t>0,
  \end{equation}
  the uniform continuity estimate
  \begin{equation}
    \label{eq:DS:35bis}
    d^2(\sfS_{t_1}(u),\sfS_{t_0}(u))\le
    2
    \sfE_{-\lambda}(t_1-t_0)\Big(\FX(\sfS_{t_0}u)-
    \FX_{\rm inf}\Big)
    \quad
    \forall\, u\in D(\FX),\ 0\le t_0\le t_1,
  \end{equation}
  and the $\lambda$-contraction property, i.e.
  \begin{equation}
    \label{eq:DS:36}
    d(\sfS_t(u),\sfS_t(v))\le e^{-\lambda t}d(u,v)\quad
    \forall\, u,v\in \Xm,\ t\ge0.
  \end{equation}
\end{prop}
\begin{proof}
  Clearly \eqref{eq:DS:13tris} yields \eqref{sdom},
  being $D(\FX)\neq\emptyset$;
  \eqref{notincr} and \eqref{eq:DS:35bis} follow
  by taking $v:=\sfS_{t_0}(u)$ and 
  \eqref{evi1} can be proved by dividing both sides of \eqref{eq:DS:13tris}
  by $t_1-t_0$ and passing to the limit
  as $t_1\downarrow t_0$.
  In order to prove the converse implication,
  let us first observe that for a continuous
  real function $\zeta:[0,+\infty)\to
  \R$
  \begin{equation}
    \label{eq:DS:29b}
    \liminf_{h\downarrow 0}\frac{\zeta(t+h)-\zeta(t)}h\le0\quad
    \forall\, t>0\qquad
    \Longrightarrow\qquad
    \text{$\zeta$ is not increasing}.
  \end{equation}
  In fact, if $0\le t_0<t_0+\tau$ existed with
  $\delta:=\tau^{-1}\big(\zeta(t_0+\tau)-\zeta(t_0)\big)>0$,
  then a minimum point $\bar t\in [t_0,t_0+\tau)$ of $
  t\mapsto \zeta(t)-\zeta(t_0)-\delta(t-t_0)$
  would satisfy
  \begin{displaymath}
    \liminf_{h\downarrow 0}\frac{\zeta(\bar t+h)-\zeta(\bar t)}h-\delta\ge
    0,\quad
    \text{which contradicts \eqref{eq:DS:29b}}.
  \end{displaymath}
  \eqref{eq:DS:13tris} then follows by \eqref{evi1}, after a
  multiplication by $e^{\lambda t}$
  and choosing
  $$\zeta(t):=
  \frac {e^{\lambda t}}2d^2(\sfS_t(u),v)
    +\int_{\bar t}^t e^{\lambda r}
    \big(\FX(\sfS_r(u))-\FX(v)\big)\,\d r,\quad \bar t>0,$$
  and recalling the monotonicity property (\ref{eq:1}b).
  A similar argument shows that
  \begin{equation}
    \label{eq:DS:13bis}
     \frac 12d^2(\sfS_{t_1}(u),v)-
     \frac 12d^2(\sfS_{t_0}(u),v)+
     \frac\lambda2\int_{t_0}^{t_1}d^2(\sfS_r(u),v)\,\d r\le
     (t_1-t_0)\Big(F(v)-F(\sfS_{t_1}(u))\Big),
   \end{equation}
   for every $0\le t_0<t_1$, $u\in \Xm$, and $v\in D(\FX)$.
  In order to prove the $\lambda$-contracting property, we apply
  \eqref{eq:DS:13bis} obtaining
  \begin{align*}
    &d^2(\sfS_h(u),\sfS_h(v))-d^2(u,v)=
    d^2(\sfS_h(u),\sfS_h(v))-d^2(\sfS_h(u),v)+
    d^2(\sfS_h(u),v)-d^2(u,v)\\&\le
    -\lambda\int_0^h \Big(d^2(\sfS_h(u),\sfS_r(v))+
    d^2(\sfS_r(u),v)\Big)\,\d r+
    2h\Big(\FX(v)-\FX(\sfS_h(v))\Big).
  \end{align*}
  We divide this inequality by $h$ and we pass to the limit as
  $h\downarrow0$;
  the
  continuity
  of $\sfS_t$, the lower semicontinuity of $\FX$, and the semigroup
  property
  of $\sfS$ yield
  \begin{equation}
    \label{eq:DS:38}
    \frac {\d^+}{\d t}d^2(\sfS_t(u),\sfS_t(v))\le -2\lambda\,
    d^2(u,v)\qquad
    \forall\, u,v\in \Xm,\ t>0,
  \end{equation}
  which yields \eqref{eq:DS:36} thanks to \eqref{eq:DS:29b}.
 \end{proof}
We can now prove the main result of this section:
if a functional $\FX$ admits a $\lambda$-flow, then
$\FX$ is geodesically $\lambda$-convex.
\begin{teo}[Geodesic convexity via E.V.I.]
  \label{convex_1}
  Let us suppose that $\sfS$ is a $\lambda$-flow
  for the functional $\FX$, according to
  \emph{(\ref{eq:1}a,b,c)},
  and let $\gamma:[0,1]\to \Xm$ be a Lipschitz curve satisfying 
  \begin{equation}
    \label{eq:DS:41}
    d(\gamma^r,\gamma^s)\le L\,|r-s|,\quad
    L^2\le d^2(\gamma^0,\gamma^1)+\eps^2
    \quad\forall\, r,s\in [0,1],
  \end{equation}
  for some constant $\eps\ge0$.
  Then for every $t>0$ and $s\in [0,1]$
  \begin{equation}
    \label{eq:DS:42}
    \FX(\sfS_t(\gamma^s))\le (1-s)\FX(\gamma^0)+s\FX(\gamma^1)-
    \frac \lambda2s(1-s)d^2(\gamma^0,\gamma^1)
    +\frac {\eps^2}{2\sfE_\lambda(t)} s(1-s).
  \end{equation}
  In particular, when $\gamma$ is a geodesic (i.e.\ $\gamma$
  satisfies \eqref{eq:DS:41} with $L=d(\gamma^0,\gamma^1),\ \eps=0$), we have
  \begin{equation}
    \label{eq:DS:12a}
    F(\gamma^s)\le (1-s)F(\gamma^0)+s F(\gamma^1)-
    \frac \lambda2s(1-s)d^2(\gamma^0,\gamma^1),
  \end{equation}
  i.e.\ $F$ is (strongly) geodesically $\lambda$-convex. 
\end{teo}
\begin{proof}
  Let $\gamma$ be satisfying \eqref{eq:DS:41} and
  let us set $\sfu^s_t:=\sfS_t(\gamma^s)$.
  Choosing $t_0=0$, $t_1=t$, $u:=\gamma^s$,
  and taking a convex combination
  of \eqref{eq:DS:13tris} written for
  $v:=\gamma^0$, and $v:=\gamma^1$, we get
  \begin{align}
    \frac{e^{\lambda t}}2 \Big((1-s)\,d^2(\sfu^s_{t},\gamma^0)
    +s\,d^2(\sfu^s_t,\gamma^1)\Big)
    &-
    \frac 12\Big((1-s)\,d^2(\gamma^s,\gamma^0)+
    s\,    d^2(\gamma^s,\gamma^1)\Big)\\&\le
    \label{eq:DS:14}
    \sfE_\lambda(t) \Big((1-s)F(\gamma^0)+sF(\gamma^1)-F(\sfu^s_t)\Big).
  \end{align}
  We now observe that the elementary inequality
  \begin{equation}
    \label{eq:DS:18}
    (1-s)a^2+sb^2\ge s(1-s)(a+b)^2\quad
    \forall\, a,b\in \R,\quad s\in [0,1],
  \end{equation}
  and the triangular inequality yield
  \begin{equation}\label{eq:DS:43}
    (1-s)d^2(\sfu^s_{t},\gamma^0)+sd^2(\sfu^s_t,\gamma^1)
    \topref{eq:DS:18}\ge
    s(1-s)\Big(d(\sfu^s_t,\gamma^0)+d(\sfu^s_t,\gamma^1)\Big)^2\ge
    s(1-s)d(\gamma^0,\gamma^1)^2.
  \end{equation}
  On the other hand, \eqref{eq:DS:41} yields
  \begin{equation}
    \label{eq:DS:17}
    (1-s)\,d^2(\gamma^s,\gamma^0)+
    s\,    d^2(\gamma^s,\gamma^1)\le
    L^2 s(1-s).
  \end{equation}
  Inserting \eqref{eq:DS:17} and \eqref{eq:DS:43} in \eqref{eq:DS:14}
  we
  obtain
  \begin{equation}
    \label{eq:DS:44}
    \frac{e^{\lambda t}-1}2 s(1-s)d^2(\gamma^0,\gamma^1)
    -\frac{\eps^2}2s(1-s)\le
    \sfE_\lambda(t) \Big((1-s)F(\gamma^0)+sF(\gamma^1)-F(\sfu^s_t)\Big).    
  \end{equation}
  Dividing then both sides of \eqref{eq:DS:44} by $\sfE_\lambda(t)$
  we get \eqref{eq:DS:42}; when $\eps=0$ we can pass
  to the limit as $t\downarrow0$ obtaining \eqref{eq:DS:12a}.
\end{proof}
We conclude this section by considering the case when the flow
$\sfS$ is only defined on a \emph{dense} subset $\Xm_0$ of $D(\FX)$.
In order to prove the geodesic convexity of $\FX$ in $\Xm$ by
Theorem \ref{convex_1} we first have
to extend $\sfS$ to the whole space $\Xm$.
This can be achieved by a density argument, if $\Xm$ is complete and the lower
semicontinuous
functional 
$\FX$ satisfies the following approximation
property:
\begin{equation}
  \label{eq:2}
  \text{$\forall\,u\in \Xm$ \quad $\exists\,u_n\in
      \Xm_0$:}\qquad \lim_{n\to\infty}d(u_n,u)=0,\quad
    \lim_{n\to\infty}\FX(u_n)=\FX(u).
  \end{equation}
We state the precise extension result in the next theorem.
\begin{teo}
  \label{thm:extension}
  Suppose that the functional $\FX$ and the subset $\Xm_0\subset
  D(\FX)$ satisfy
  \eqref{eq:2} and let $\sfS$ be a $\lambda$-flow for $\FX$ in
  $\Xm_0$.
  If $\Xm$ is complete,
  $\sfS$ can be extended to a unique $\lambda$-flow $\bar\sfS$
  in $\Xm$ and therefore $\FX$ is (strongly) geodesically $\lambda$-convex in $\Xm$.
\end{teo}
\begin{proof}
  Given $u\in \Xm$ and a sequence $u_n\in \Xm_0$
  as in (\ref{eq:2}), we can define
  \begin{equation}
    \label{eq:DS:23}
    \bar\sfS_t(u):=\lim_{n\to\infty}\sfS_t(u_n)\quad
    \forall\, t>0,
  \end{equation}
  where it is clear that the limit in \eqref{eq:DS:23} exists
  (being $X$ complete and $\sfS_t$ Lipschitz by
  \eqref{eq:DS:36})
  and does not depend on the particular
  sequence $u_n$ we used to approximate $u$. Moreover
  $\bar\sfS_t$ is a semigroup and satisfies the estimate
  \eqref{eq:DS:35bis} and the
  $\lambda$-contracting property \eqref{eq:DS:36}; being $D(\FX)$ dense
  in $\Xm$, it is not difficult to combine \eqref{eq:DS:35bis}, \eqref{eq:DS:36} and \eqref{eq:2} to show that
  $\lim_{t\downarrow0}\sfS_t(u)=u$ for every $u\in \Xm$.

  In order to prove that $\bar\sfS$ is still a $\lambda$-flow for
  $\FX$ in $\Xm$ we have to check \eqref{eq:DS:13tris} in $\Xm$:
  we fix $v\in D(\FX)$ and a sequence $v_n\in \Xm_0$
  converging to $v$ with $\FX(v_n)\to\FX(v)$ and we pass to the
  limit as $s\to\infty$
  in the inequalities
  \begin{equation}
    \label{eq:DS:25}
    \frac {e^{\lambda (t_1-t_0)}}2 d^2(\sfS_{t_1}(u_n),v_n)-
    \frac 1
    2 d^2(\sfS_{t_0}(u_n),v_n)
    \le
    \sfE_\lambda(t_1-t_0)\big(\FX(v_n)-\FX(\sfS_{t_1}(u_n)),
  \end{equation}
  using the lower semicontinuity of $\FX$.
\end{proof}

\section{Nonlinear diffusion equations as gradient flows of entropy functionals in $\PP_2(\M)$}
We apply the strategy described in the Section 2 to prove the
geodesic convexity of the integral functional \eqref{entropintro}
in the case of a Riemannian manifold of nonnegative Ricci curvature.
We therefore exhibit a smooth flow (induced by the nonlinear diffusion
equation \eqref{eq0} on the dense subset
$\PP_2^{ar}(\M)$) which satisfies the Evolution Variational Inequality
\eqref{flussograd0}.

Before stating the main theorem of this section let us
recall a fundamental result on this kind of evolution equations,
that can be found in \cite{Vaz1, OttoW}:
\begin{teo}[Classical solutions of nonlinear diffusion equations]\label{classiche}
  Let $\U\in C^\infty(\R^+)$ and $\pr$
  be functions that satisfy the assumptions \eqref{1}
  of Theorem \ref{teo0}. 
  For every $\rhozero\in C^\infty(\M)$ with
  $\rhozero>0$, there exists a unique smooth positive solution
  $\rho\in C^\infty([0,+\infty)\times\Mm)$
  to the Cauchy problem
  \begin{equation}\label{eq1}
    \partial_t\rho_t=\lapB \pr(\rho_t),\qquad
    \rho\restr{t=0}=\lim_{t\downarrow0}\rho_t=\rhozero.
  \end{equation}
  Moreover, given a one parameter family of positive 
  initial data $s\mapsto\rho^s_0\in
  C^{\infty}([0,1]\tildina\times\tildina\M)$,
  the corresponding solutions $\rho^s_t$
  of the equation \eqref{eq1}
  depend smoothly on $s,t$.

  For every $\mu_0=\rhozero\Vol\in \PP_2^{ar}(\M)$
  we denote by $\SS_t(\mu_0)\in\PP_2^{ar}(\M)$ the
  measure $\mu_t=\rho_t\Vol$.
\end{teo}
The main result that we show in this section is the following:
\vskip 0.2 cm
\begin{teo}\label{teorema 1}
Let $\U\in C^\infty(\R^+)$ and $\pr$ be functions that satisfy the assumptions \eqref{1} of Theorem \ref{teo0} and let us suppose that
\begin{equation}
\mathrm{Ric}(x)\geq0\quad\forall\,x\in\M.\label{Ric1}
\end{equation}
The semigroup $\SS$ induced by \eqref{eq1} in $\PP_2^{ar}(\M)$
is a $0$-flow in $\PP_2^{ar}(\M)$ for the functional
\begin{equation}\label{entropia62}
\FF(\mu)=\int\nolimits_{\M}\U(\rho)\,\d\Vol, \quad \forall\, \mu=\rho\Vol\in \PP_2^{ar}(\M).
\end{equation}
In particular, for every $\mu_0=\rho_0\Vol,\nu\in \PP_2^{ar}(\M)$,
the measures $\mu_t=\SS_t(\mu_0)=\rho_t\Vol\in \PP_2^{ar}(\M)$ solving
\eqref{eq1}
satisfy the E.V.I.
\begin{equation}\label{derivata distanza bis}
  \frac{1}{2}\frac{\d^+}{\d t} W_2^2(\nu,\mu_t)\leq\FF(\nu)-\FF(\mu_t)
  \quad\forall\,t\in[0,+\infty).
\end{equation}
\end{teo}
In order to prove Theorem \ref{teorema 1}, 
thanks to the ``Riemannian-like'' characterization
of the Wasserstein distance provided by \eqref{riemann0},
we can follow the
strategy
presented in Section 2, in particular we want to prove the
differential inequality of Theorem \ref{thm:main_smooth}.
Following \textsc{Otto's} formalism, we collect in the next table
the formal correspondences between the various objects:
\begin{center}
  \begin{tabular}[c]{c|c}
    $\Xm$, Riemannian manifold, with distance $d$&
    $\PP_2^{ar}(\M)$ with distance $W_2$\vspace{2pt}\cr
    \hline
    a smooth curve $\gamma^s$ 
    in $\Mm$&
    a smooth family $\mu^s=\rho^s\Vol\in \PP_2^{ar}(\M)$
    \\
    the tangent vector
    $\partial_s\gamma^s$ in $T_{\gamma^s}\Mm$&
    the vector field $\nabla \phi^s$ where
    $-\nabla\cdot(\rho^s\nabla\phi^s)=\frac\partial{\partial s} \rho^s$\\
    $\big|\partial_s\gamma^s\big|_g^2$&
    $\displaystyle
    \int\nolimits_\M \big|\nabla\phi^s(x)\big|_\sfg^2\,
    \rho^s(x)\,\d\Vol(x)$
    \\ 
    $\gamma^s_t:=\sfS_t(\gamma^s)$, 
    $\tilde\gamma^s_t:=\gamma^s_{st}=\sfS_{st}(\gamma^s)$ &
    $\mu^s_t=\rho^s_t\,\Vol:=\SS_t(\mu^s)$,
    $\tilde\mu^s_t=
    \tilde\rho^s_t\,\Vol:=\mu^s_{st}=\SS_{st}(\mu^s)$\\
    $\tilde A^s_t=\big|\partial_s \tilde \gamma^s_t\big|_g^2$&
    $\displaystyle
    \int\nolimits_\M  \big|\nabla\tilde\phi^s_t(x)\big|_\sfg^2\,
    \tilde\rho^s_t(x)\,\d\Vol(x)$\\
    $F(\gamma^s)$&
    $\displaystyle\FF(\mu^s)=\int\nolimits_{\M} \U(\rho^s)\,\d\Vol$\\
    $\big(\partial_\theta\sfS_\theta \gamma^s\big)\restr{\theta=0}=
    -\nabla F(\gamma^s)$&
    $-\nabla \pr(\rho^s)/\rho^s=-\nabla \U'(\rho^s)$.
  \end{tabular}
\end{center}
The core of the proof of Theorem \ref{teorema 1}
lies in the following lemma:
\begin{lemma}\label{derivazione}
  Let $\mu^s=\rho^s\Vol$, $s\in [0,1]$, be a smooth family of measures
  in $ \PP_2^{ar}(\M)$ and let $\tilde\mu^s_t=\tilde\rho^s_t\Vol
  =\SS_{st}(\mu^s)$
  be obtained by flowing $\rho^s$ along the flow \eqref{eq1},
  i.e.\ $\tilde\rho^s_t=\rho^s_{st}$ where
  $\rho^s_t$
  satisfies
  \begin{equation}\label{ip2}
    \frac\partial{\partial t}\rho^s_t-\lapB \pr(\rho^s_{t})=0\text{ in }\M,\quad
    \forall\,s\in[0,1],\ t>0;\qquad
    \rho^s_{t=0}=\rho^s.
  \end{equation}
  Let $\tilde\phi^s_t\in
  C^{\infty}([0,1]\times[0,+\infty)\times\M)$ be
  the functions defined by the equation 
  \begin{equation}
    \label{eq:DS:29c}
    -\nabla\cdot(\tilde\rho^s_t\nabla\tilde\phi^s_t)=\partial_s\tilde\rho^s_t
    \quad\text{ in }\M,\quad
    \int_\M \tilde\phi^s_t(x)\,\d\Vol(x)=0
    \quad\forall\,s\in [0,1],\ t\in[0,+\infty),
  \end{equation}
  and let us set
  \begin{equation}
    \label{eq:DS:30b}
    \begin{aligned}
      \tilde A^s_t:=&
      \int_{\M}|\nabla\tilde\phi^s_t(x)|_{\sfg}^2\,\tilde\rho^s_t(x)\,
      \d\Vol(x),\\
      \tilde D^s_t:=&
      -\int_{\M}\bigg[\Big(|\mathrm{Hess}\,\tilde\phi^s_t|^2_{\sfg}+ 
      \mathrm{Ric}\,(\nabla\tilde\phi^s_t,\nabla\tilde\phi^s_t)\Big)\,\pr\,(\tilde\rho^s_t)
      +(\lapB\tilde\phi^s_t)^2\,
      \big(\tilde\rho^s_t \pr'(\tilde\rho^s_t)-\pr\,(\tilde\rho^s_t)\big)\bigg]
      \d\Vol.
    \end{aligned}
  \end{equation}
Then, we have the formula
\begin{align}\label{derivata azione}
  \frac{\partial}{\partial t}\frac 12\tilde A^s_t+
  \frac{\partial}{\partial s}\FF(\tilde\rho^s_t \Vol)
  =s\tilde D^s_t,
  \quad\forall\,t\in[0,+\infty),\,\forall\,s\in[0,1].
\end{align}
In particular, if $\M$ has nonnegative Ricci curvature, then $\tilde
D^s_t\le0$
and therefore
\begin{equation}
  \label{eq:DS:31a}
  \frac{\partial}{\partial t}\frac 12\tilde A^s_t+
  \frac{\partial}{\partial s}\FF(\tilde\rho^s_t \Vol)\le0.
\end{equation}
\end{lemma}
\begin{proof}
Being $\tilde\rho^s_t:=\rho^\sigma_{\tau}\restr{\sigma=s,\tau=st}$
we get
\begin{align}
  \label{eq:DS:30c}
  \tfrac\partial{\partial s}\tilde\rho^s_t&
  =\Big(\tfrac \partial{\partial \sigma}
  \rho^{\sigma}_{\tau}+t\tfrac\partial{\partial\tau}\rho^\sigma_{\tau}
  \Big)_{\sigma=s,\tau=st},\qquad
  \tfrac\partial{\partial t}\tilde \rho^s_t
  =s\partial_\tau\rho^s_{\tau}\restr{\tau=st}=s\lapB \pr\,(\tilde\rho^s_t),\\
  \tfrac{\partial^2}{\partial t\,\partial s}\tilde\rho^s_t
  &\topref{eq:DS:29c}
  =-\nabla\cdot(\tfrac\partial{\partial t}\tilde\rho^s_t\,
  \nabla\tilde\phi^s_t)-
  \nabla\cdot(\tilde\rho^s_t\,
  \tfrac\partial{\partial t}\nabla\tilde\phi^s_t),\label{2.45}
  \\
  \label{2.46}
  \tfrac{\partial^2}{\partial s\,\partial t}\tilde\rho^s_t&
  \topref{eq:DS:30c}=
  s\lapB\Big(
  \pr'(\tilde\rho^s_t)\,\tfrac{\partial}{\partial s}\tilde\rho^s_t\Big)
  +\lapB \,\pr\,(\tilde\rho^s_t)
  \topref{eq:DS:29c}=
  -s\lapB\Big(\pr'(\tilde\rho^s_t)\,\nabla\cdot(\tilde\rho^s_t\nabla\tilde\phi^s_t)
  \Big)+\lapB 
\pr\,(\tilde\rho^s_t).
\end{align}
Differentiation and integration by parts yield
\begin{align}\label{calcolo}
  \frac{\partial}{\partial
    t}&\int\nolimits_{\M}\frac{1}{2}|\nabla\tilde\phi^s_t|_\sfg^2\,
  \tilde\rho^s_t\,
  \d\Vol =
  \int\nolimits_{\M}
  \la\tfrac\partial{\partial t}\nabla\tilde\phi^s_t,
  \nabla\tilde\phi^s_t\ra_\sfg\,\tilde\rho^s_t\,\d\Vol
  +
  \tfrac{1}{2}\int\nolimits_{\M}|\nabla\tilde\phi^s_t|_\sfg^2\,
  \tfrac{\partial}{\partial t}\tilde\rho^s_t\, \d\Vol =\notag\\
  &=-\int\nolimits_{\M}
  \nabla\cdot(\tilde\rho^s_t\tfrac\partial{\partial t}\nabla\tilde\phi^s_t)\,\tilde\phi^s_t\,
  \d\Vol
  \topref{eq:DS:30c}+
  \frac{1}{2}s\int\nolimits_{\M}
  \lapB(|\nabla\tilde\phi^s_t|_\sfg^2)\,\pr(\tilde\rho^s_t)\, \d\Vol =\notag\\
  &\overset{\eqref{2.45}}{=}
  \int\nolimits_{\M}\tfrac {\partial^2}{\partial t\partial s}
  \tilde\rho^s_t \, \tilde\phi^s_t\,\d\Vol
  +
  \int\nolimits_{\M}\Big(
  \nabla
  \cdot(\tfrac\partial{\partial t}\tilde\rho^s_t\nabla\tilde\phi^s_t)
  \Big)\,\tilde\phi^s_t\,\d\Vol
  +
  \frac{1}{2}s\int\nolimits_{\M}\lapB(|\nabla\tilde\phi^s_t|_\sfg^2)\,
  \pr(\tilde\rho^s_t)\,
  \d\Vol =\notag\\
  &\overset{\eqref{2.46}}{=}
  \int\nolimits_{\M}\Big(
  \lapB \pr\,(\tilde\rho^s_t)-
  s
  \lapB\Big(\pr'(\tilde\rho^s_t)\,\nabla\cdot(\tilde\rho^s_t\nabla\tilde\phi^s_t)
  \Big)\,
  \tilde\phi^s_t\, \d\Vol
  \\&\qquad
  -s\int\nolimits_{\M}\lapB
  \pr(\tilde\rho^s_t)\,|\nabla\tilde\phi^s_t|_\sfg^2\, \d\Vol
  +\frac s2\int\nolimits_{\M}
  \lapB\Big(|\nabla\tilde\phi^s_t|_\sfg^2\Big)\,\pr(\tilde\rho^s_t)\,
  \d\Vol =\notag\\
  &=
  \int\nolimits_{\M}\pr(\tilde\rho^s_t)\,\lapB\tilde\phi^s_t\,
  \d\Vol -
  s\int\nolimits_{\M}\Big(
  \big\la\nabla  \pr(\tilde\rho^s_t),\nabla\tilde\phi^s_t\big\ra_\sfg\,
  \lapB \tilde\phi^s_t
  +
  \tilde\rho^s_t\,\pr'(\tilde\rho^s_t)\big(\lapB
  \tilde\phi^s_t\big)^2
  \Big)\,\d\Vol\notag\\
  &\qquad
  -\frac s2
  \int\nolimits_{\M}
  \lapB\big(|\nabla\tilde\phi^s_t|_\sfg^2\big)\, \pr(\tilde\rho^s_t)\,\d\Vol
  \notag\\
 &=
 -\int\nolimits_{\M} \big\la\nabla
  \pr\,(\tilde\rho^s_t),\nabla\tilde\phi^s_t\big\ra_\sfg \,
  \d\Vol +
  s\int\nolimits_{\M}\biggl[-\frac{1}{2}\lapB(|\nabla\tilde\phi^s_t|_\sfg^2)+
  \la\nabla\tilde\phi^s_t,\nabla\lapB\tilde\phi^s_t\ra_\sfg\biggr]
  \pr\,(\tilde\rho^s_t)\,\d\Vol +\notag\\
  &\qquad+s\int\nolimits_{\M}\big(\lapB\tilde\phi^s_t\big)^2
  \,\Big(\pr\,(\tilde\rho^s_t)-\tilde\rho^s_t
\pr'(\tilde\rho^s_t)\Big)\,\d\Vol  
\end{align}
Applying Bochner formula:
\begin{equation}
\la\nabla\phi,\nabla\lapB\phi\ra_\sfg -
\tfrac{1}{2}\lapB\big(|\nabla\phi|_\sfg^2\big)=-|\mathrm{Hess}\,\phi|_\sfg^2 -
\mathrm{Ric}(\nabla\phi,\nabla\phi),
\end{equation}
we get 
\begin{equation}\label{eq121}
  \frac{\partial}{\partial t}
  \frac{1}{2}
  \int_{\M}|\nabla\tilde\phi^s_t|_\sfg^2\,\tilde\rho^s_t\,
  \d\Vol +\int_{\M}\big\la\nabla
  \pr(\tilde\rho^s_t),\nabla\tilde\phi^s_t\big\ra_\sfg \,
  \d\Vol =s\tilde D^s_t.
\end{equation}
Now we observe that the second term in
the right-hand side of \eqref{eq121} is the derivative of the functional
\eqref{entropia62} along the curve $s\mapsto\tilde\rho^s_tV\in\PP_2^{ar}(\M)$:
\begin{equation}
  \frac{\partial}{\partial s}\FF(\tilde\mu^s_t)
  =\int_{\M}\U'(\tilde\rho^s_t)
  \,\tfrac\partial{\partial s}\tilde\rho^s_t\,\d\Vol =-\int_{\M}\U'(\tilde\rho^s_t)\nabla\cdot(\tilde\rho^s_t\nabla\tilde\phi^s_t)\,\d\Vol 
=\int_{\M}\nabla \pr\,(\tilde\rho^s_t)\cdot\nabla\tilde\phi^s_t\,\d\Vol
\end{equation}
and we eventually obtain \eqref{derivata azione}.

Finally, when $\textrm{Ric}(\M)\ge0$, using the inequality
$(\lapB\phi)^2\leq n|\mathrm{Hess}\,\phi|_\sfg^2$ and \eqref{1} we
easily get $\tilde D^s_t\le 0$ and \eqref{eq:DS:31a}.
\end{proof}
\emph{Proof of Theorem \ref{teorema 1}.}
We argue as in the proof of Theorem \ref{thm:main_smooth}:
we fix $\varepsilon>0$
and
we choose a smooth curve $(\rho,\phi)\in \CC(\nu,\mu)$
such that 
\begin{equation}\label{epsilon-geodetica}
  \int\nolimits_0^1 \tilde A^s_0\,\d s=
  \int\nolimits_0^1\int\nolimits_{\M}
  |\nabla\phi^s|_\sfg^2\, \rho^s\,\d\Vol \d s
  \leq W_2^2(\nu,\mu)
  +\varepsilon.
\end{equation}
Let $(\tilde\rho,\tilde\phi)$ a smooth variation defined 
as in Lemma \ref{derivazione}; since
$\tilde\rho^0_t\Vol=\rho^0\Vol=\nu$ and $\tilde\rho^1_t\Vol=\mu_t$,
for every $t>0$ we have $(\tilde\rho^s_t,\tilde\phi^s_t)\in \CC(\nu,\mu_t)$
and therefore
\begin{equation}
  \label{eq:DS:29d}
  W^2_2(\nu,\mu_t)\le \int\nolimits_0^1 \int\nolimits_\M
  |\nabla\tilde\phi^s_t|_\sfg^2 \,\tilde\rho^s_t\,\d\Vol\, \d s=
  \int\nolimits_0^1 \tilde
  A^s_t\,\d s.
\end{equation}
Integrating \eqref{eq:DS:31a} for $s\in [0,1]$ and $t\in [0,\tau]$ and
recalling that $t\mapsto \FF(\mu_t)$ is not increasing, we
get
\begin{equation}
  \label{eq:DS:30d}
  \frac 12 \int\nolimits_0^1 \tilde A^s_\tau\,\d s-
  \frac 12 \int\nolimits_0^1 \tilde A^s_0\,\d s\le 
  \tau\Big(\FF(\nu)-\FF(\mu_\tau)\Big).
\end{equation}
Combining \eqref{eq:DS:30d} with \eqref{eq:DS:29d}
and \eqref{epsilon-geodetica} we get
\begin{equation}\label{0.10}
\frac{1}{2}W_2^2(\nu,\mu_\tau) - \frac{1}{2}
W_2^2(\nu,\mu)
\le
 \tau\Big(\FF(\nu)-\FF(\mu_\tau)\Big)+\varepsilon,
\end{equation}
and, as $\varepsilon$ is arbitrary,
\begin{equation}\label{der1}
\frac{1}{2}W_2^2(\nu,\mu_\tau) - \frac{1}{2}
W_2^2(\nu,\mu)
\le
 \tau\Big(\FF(\nu)-\FF(\mu_\tau)\Big).
\end{equation}
Since the semigroup associated to \eqref{eq1} is translation
invariant, \eqref{der1} is the integral formulation
\eqref{eq:DS:13tris}
of \eqref{derivata distanza bis}.\quad$\Box$
\begin{remark}
  \label{rem:Otto}
  \upshape
  Taking into account Theorem \ref{thm:main_reduction},
  \eqref{derivata azione} perfectly fits with the calculation
  performed by
  \cite[Lemma 4.4]{OttoW}, which 
  provides the same expression
  for $\tilde D^s_t$.
\end{remark}
Applying now
Theorem \ref{thm:extension}, with the choices $\Xm:=\PP_2(\M)$,
$\Xm_0:=\PP_2^{ar}(\M)$, $\FX:=\FF$ (which satisfies
the approximation condition \eqref{eq:2}, see \cite{Ambrosio-Buttazzo88})
we can prove the first part
of Theorem \ref{teo0}. 
\begin{corollario}\label{teorema 2}
  Let $\FF:\PP_2(\M)\freccia(-\infty,+\infty]$ be the
  functional defined in \eqref{entropintro}.
  If $\U$ satisfies \textsc{McCann} conditions
  \eqref{1} and
  $\textrm{Ric}(\M)\ge0$, then
  $\FF$ is (strongly) displacement convex along \emph{every} geodesic
  $\mu:s\in [0,1]\mapsto \mu^s\in \PP_2(\M)$, i.e.
  \begin{equation}\label{convessita geodetica}
    \FF(\mu^s)\leq(1-s)\FF(\mu^0)+s\FF(\mu^1)\quad\forall\,s\in[0,1].
\end{equation}
\end{corollario}


\section{The Heat equation and the displacement $\lambda$-convexity of
  the logarithmic Entropy}
In this last section we prove the second part of Theorem \ref{teo0}:
we thus assume that the Riemannian manifold $\M$ satisfies the
lower Ricci curvature bound
\begin{equation}
  \label{Ric1lambda}
  \mathrm{Ric}(\M)\ge \lambda\quad\text{i.e.}\quad
  \mathrm{Ric}_x(\xi,\xi)\ge \lambda\,|\xi|_\sfg^2\quad
  \forall\, \xi\in T_x\,\M,
\end{equation}
and we consider the logarithmic entropy functional
\begin{equation}\label{entropiab}
\FF(\mu)=\int\nolimits_{\M}\rho\log\rho\,\d\Vol , \quad \rho=\frac{d\mu}{\d\Vol },
\end{equation}
corresponding to $\U(\rho):=\rho\log\rho$.
Since $\pr(\rho)=\rho$, the Wasserstein gradient flow associated
to $\FF$ is the Heat equation
\begin{equation}
  \label{heat}
  \frac\partial{\partial {t}}\rho_t - \lapB \rho_t =
  0\quad\text{ in }\M,\qquad
  \rho\restr{t=0}=\rho_0.
\end{equation} 
The main result of this section is the following:
\begin{teo}\label{teorema 3}
  The semigroup $\SS_t:\mu_0=\rho_0\Vol\mapsto \mu_t=\rho_t\Vol$,
  generated by the solution of the Heat equation
  \eqref{heat} is a $\lambda$-flow in $\PP_2^{ar}(\M)$
  for the logarithmic entropy functional, i.e.\
  $\mu_t$  satisfies the inequality
\begin{equation}\label{derivata distanza bisl}
  \frac{1}{2}\frac{\d^+}{\d t}W_2^2(\nu,\mu_t)+
  \frac{\lambda}{2}W_2^2(\nu,\mu_t)\leq\FF(\nu)-\FF(\mu_t)
  \quad\forall\,t\in[0,+\infty),\ \nu\in \PP_2^{ar}(\M).
\end{equation}
In particular,
the logarithmic entropy functional \eqref{entropiab}
is (strongly) displacement $\lambda$-convex, i.e.\
for every geodesic $\mu^s:[0,1]\freccia\PP_2(\M)$ 
between $\mu^0$ and $\mu^1$, we have
\begin{equation}\label{convessita geodetica2}
  \FF(\mu^s)\leq(1-s)\FF(\mu^0)+s\FF(\mu^1)
  -\frac{\lambda}{2}s(1-s)W_2^2(\mu^0,\mu^1),\quad\forall\,s\in[0,1].
\end{equation}
\end{teo}
\begin{proof}
  By Theorem \ref{thm:extension},
  if $\SS$ is a $\lambda$-flow for the
  functional \eqref{entropiab} in $\PP_2^{ar}(\M)$ then 
  $\FF$ is (strongly) displacement $\lambda$-convex.
  In order to prove that $\SS$ is a $\lambda$-flow, since
  (\ref{eq:1}a,b) are immediate, we check that $\SS$
  satisfies the E.V.I.\ 
  \eqref{evi1} and
  we argue as in the proof of Theorem \ref{teorema 1}
  and Theorem \ref{thm:main_smooth}.
  We thus fix $\eps>0$ and we choose a smooth curve
  $(\rho,\phi)\in \CC(\nu,\mu)$
  \begin{equation}
    \label{epsilon-geodetica2}
    \int\nolimits_0^1 \tilde A^s_0\,\d s=
    \int\nolimits_0^1\int\nolimits_{\M}
    |\nabla\phi^s|_\sfg^2\, \rho^s\,\d\Vol \d s
    \leq W_2^2(\nu,\mu)
    +\varepsilon^2.
  \end{equation}
  By a standard re-parametrization technique (see next Lemma \ref{0}),
  we can also
  assume that
  \begin{equation}
    \label{eq:DS:13}
    W_2(\mu^{s_0},\mu^{s_1})\le L|s_0-s_1|,\quad
    L^2:=W_2^2(\nu,\mu)
    +\varepsilon^2\qquad
    \forall\, s_0,s_1\in [0,1];\qquad
    \mu^s:=\rho^s\,\Vol.
  \end{equation}
  We keep the same notation of Theorem \ref{teorema 1}
  and Lemma \ref{derivazione},
  i.e.
  \begin{equation}
    \label{eq:DS:5}
    \tilde\mu^s_t=\tilde\rho^s_t\,\Vol:=\SS_{st}(\mu^s),\quad
    \tilde A^s_t:=\int_{\M}|\nabla\tilde\phi^s_t|_\sfg^2\,
    \tilde\rho^s_t\,
    \d\Vol,\quad
    \tilde F^s_t =\FF(\tilde\mu^s_t)
  \end{equation}
  where $\tilde\phi^s_t$ is family of potentials associated to
  $\tilde\rho^s_t$ as in \eqref{eq:DS:29c}.
  Since $\pr(\rho)=\rho$ the term $\rho \pr'(\rho)-\pr(\rho)$ in the
  definition
  of $\tilde D^s_t$ vanishes, so that in the present case
  \begin{equation}
    \label{eq:DS:29e}
    \tilde D^s_t=-\int\nolimits_{\M}\Big(|\mathrm{Hess}\,\tilde\phi^s_t|_\sfg^2 +
    \mathrm{Ric}\,(\nabla\tilde\phi^s_t,\nabla\tilde\phi^s_t)
    \Big)\,\tilde\rho^s_t\,\d\Vol
    \topref{Ric1lambda}\le -\lambda
    \int\nolimits_{\M}|\nabla\tilde\phi^s_t|_\sfg^2\tilde\rho^s_t\,
    \d\Vol =-\lambda \tilde A^s_t,
  \end{equation}
  \eqref{derivata azione} yields
  the differential inequality
  \begin{equation}\label{ineq}
   \frac 12 \frac\partial{\partial t}\tilde A^s_t +\lambda s\tilde A^s_t +
    \frac\partial{\partial s}\tilde F^s_t \leq0\qquad \forall\,s\in[0,1],\quad\forall\,t>0.
  \end{equation}
Multiplying inequality \eqref{ineq} by $e^{2\lambda s t}>0$ we obtain
\begin{equation}
  \label{eq:step:4}
  \frac 12\frac\partial{\partial t} \Big(e^{2\lambda s t}\tilde A^s_t \Big)+
  \frac \partial{\partial s}\Big(e^{2\lambda s t}\tilde F^s_t \Big)\le 2\lambda t\,e^{2\lambda
    st}\,\tilde F^s_t.
\end{equation}
Integrating with respect to $s$ from $0$ to $1$ we get
\begin{equation}
  \label{eq:step:6}
  \frac \d{\d t}\Big(\frac 12\int_0^1 e^{2\lambda st}\tilde A^s_t \,\d s\Big)+
  e^{2\lambda t}\tilde F^1_t-\tilde F^0_t\leq\int_0^1 2\lambda\, t\,e^{2\lambda
    st}\tilde F^s_t \,\d s,
\end{equation}
and a further integration with respect to $t$ yields
\begin{equation}
  \label{eq:step:7}
  \frac 12\int_0^1 e^{2\lambda st}\tilde A^s_t \,\d s-
  \frac 12\int_0^1 A^s_0\,\d s+
  \sfE_{2\lambda}(t)\FF(\mu_t)-t\FF(\nu)
  \leq\int_0^t\int_0^1 2\lambda \,r\,e^{2\lambda
    sr}\,\tilde F^s_r\,\d s\,\d r.
\end{equation}
Applying the next Lemma \ref0,
since for $\lambda \neq 0$
$
  \int_0^1 \frac 1{e^{2\lambda st}}\,\d s=\frac{1-e^{-2\lambda
      t}}{2\lambda t}=\frac 1{e^{\lambda t}\sfs(\lambda t)},\quad
  \sfs(t):=\frac t{\sinh(t)},
  $
we get 
\begin{align}
  \label{eq:step:8}
  \frac {e^{\lambda t}\sfs(\lambda t)}{2}W_2^2(\mu_t,\nu)-
  \frac 12 W_2^2(\mu,\nu)+
   \sfE_{2\lambda}(t)\FF(\mu_t)-t\FF(\nu)
  \leq\int_0^t\int_0^1 2\lambda re^{2\lambda
    sr}\tilde F^s_r\,\d s\d r+\frac{\varepsilon^2}2.
\end{align}
\underline{Let us first consider the case $\lambda<0$}:
being $\FF$ nonnegative, 
the right hand side in \eqref{eq:step:8}
is less or equal than $\varepsilon$; since
$\varepsilon>0$ is arbitrary, we obtain the same inequality
with $0$ in the right-hand side.
Since $t^{-1}\sfE_{2\lambda}(t)\to1$ as $t\downarrow0$ and
$\sfs(0)=1$, we thus obtain
\begin{equation}
  \label{eq:step:9}
  \frac 12\frac {\d^+}{\d t}\Big({e^{\lambda t}\sfs(\lambda t)}W_2^2(\mu_t,\nu)
  \Big)\Big|_{t=0}+\FF(\mu)\le \FF(\nu).
\end{equation}
Being $\sfs'(0)=0$ it is then easy to check that
\begin{displaymath}
  \frac {\d^+}{\d t}\Big({e^{\lambda t}\sfs(\lambda t)}W_2^2(\mu_t,\nu)
  \Big)\Big|_{t=0}=
  \frac {\d^+}{\d t}\Big(W_2^2(\mu_t,\nu)
  \Big)\Big|_{t=0}+
  \lambda\, W_2^2(\mu,\nu),
\end{displaymath}
which yields \eqref{derivata distanza bisl}.

\noindent
\underline{Let us now consider the case $\lambda>0$}.
By \eqref{eq:DS:13} we can apply the estimate
\eqref{eq:DS:42} obtaining
\begin{align*}
  r\tilde F^s_r&=   r\FF(\SS_{rs}(\mu^s))
  \topref{eq:DS:42}\le r\Big((1-s)\FF(\mu^0)+s\FF(\mu^1)-\frac\lambda2 s(1-s)W_2^2(\mu^0,\mu^1)+
  \frac{\eps^2}{2\sfE_\lambda(rs)}s(1-s)\Big)\\
  &\le   r\Big(\FF(\mu^0)+\FF(\mu^1)\Big)+\eps^2,
\end{align*}
since $s\in [0,1]$ and $rs/\sfE_{\lambda}(rs)\le 1$. We thus get
\begin{equation}
  \label{eq:DS:30e}
  \int_0^t\int_0^1 2\lambda \, r\,e^{2\lambda s r}\tilde F^s_r\,\d s\,\d r\le
  2\lambda t e^{2\lambda t}\Big(t\big(\FF(\mu_0)+\FF(\mu_1)\big)+\eps^2\Big);
\end{equation}
inserting this bound in \eqref{eq:step:8} and passing to the limit as
$\eps\downarrow0$ 
we find
\begin{equation}\label{eq100}
 \frac {e^{\lambda t}\sfs(\lambda t)}{2}W_2^2(\mu_t,\nu)-
  \frac 12 W_2^2(\mu,\nu)+
   \sfE_{2\lambda}(t)\FF(\mu_t)-t\FF(\nu)
  \leq  2\lambda t^2 e^{2\lambda t}\Big(\FF(\mu_0)+\FF(\mu_1)\Big).
\end{equation} 
Dividing by $t$ and letting $t$ tend to 0 the second term vanishes, so
we obtain the EVI also in the case in which $\lambda>0$.
\end{proof}
\begin{lemma}\label{0}
  Let $\nu,\mu\in \PP_2^{ar}(\M)$ and
  let $(\rho,\phi)\in \CC(\nu,\mu)$ be
  a smooth solution of the continuity equation
  \begin{displaymath}
    \frac\partial{\partial s}
    \rho^s+\nabla\cdot(\rho^s\,\nabla\phi^s)=0
    \quad\text{in }
    [0,1]\times \M\quad\text{with}\quad \rho^0\Vol=\nu,\:\rho^1\Vol=\mu\quad\text{and} \quad A^s:=\int_\M |\nabla\phi^s|_\sfg^2\,\rho^s\,\d\Vol.
  \end{displaymath}
  For every positive function $f\in C^\infty[0,1]$
  \begin{equation}
    \label{eq:step:1}
    W_2^2(\nu,\mu)\le L_f
    \int_0^1 f(s)A^s\,ds,\quad
    \text{where}\quad
    L_f:=\int_0^1 \frac 1{f(s)}\,ds.
  \end{equation}
  Moreover, for every $\eps>0$ there exists a smooth rescaling
  $\sfs_\eps:[0,1]\to[0,1]$ so that the re-parametrized families
  \begin{equation}
    \label{eq:DS:16}
    \bar\rho^r:=\rho^{\sfs_\eps(r)},\quad
    \bar\phi^r:=\sfs_\eps'(r)\phi^{\sfs_\eps(r)},\quad
    \bar\mu^r:=\bar\rho^r\,\Vol
  \end{equation}
  satisfy
  \begin{equation}
    \label{eq:DS:29}
    (\bar\rho,\bar\phi)\in \CC(\nu,\mu),\quad
    W_2(\bar\mu^{r_0},\bar\mu^{r_1})\le L|r_0-r_1|,\quad
    L^2\le \int_0^1 A^s\,\d s+\eps^2.
  \end{equation}
\end{lemma}
\begin{proof}
  Let us consider the smooth increasing map
  ${\sf r}:[0,1]\to[0,1]$
  \begin{displaymath}
    {\sf r}(s):=L_f^{-1}\int\nolimits_0^s \frac 1{f(s)}\,ds
    \quad
    \text{and its inverse }
    {\sf s}:={\sf r}^{-1}\quad
    \text{with}\quad
    {\sf s}'({\sf r}(s))=L_f f(s).
  \end{displaymath}
  It is immediate to check that
  the smooth (reparametrized) curve
  \begin{equation}
    \label{eq:step:2}
    \bar\rho^r(x):=\rho^{{\sf s}(r)}(x),\quad
    \bar \phi^r(x):= {\sf s}'(r)\phi^{{\sf s}(r)}(x)
  \end{equation}
  belongs to $\CC(\nu,\mu)$.
  It follows that
  \begin{displaymath}
    W_2^2(\nu,\mu)\le \int\nolimits_0^1 \bar A^r\, \d r,\quad
    \text{where}\quad
    \bar A^r:=\int\nolimits_{\M}|\nabla\bar\phi^r|_\sfg^2\,
    \bar\rho^r
    \,\d\Vol
    \topref{eq:step:2}= \big({\sf s}'(r)\big)^2 A^{{\sf s}(r)},
  \end{displaymath}
  so that
  \begin{displaymath}
    \int\nolimits_0^1 \bar A^r\d r =
    \int\nolimits_0^1 A^{{\sf s}(r)}\big({\sf s}'(r)\big)^2\,\d r =
    \int\nolimits_0^1 A^s{\sf s}'({\sf r}(s))\,\d s =
    L_f \int\nolimits_0^1 f(s)A^s\,\d s.
  \end{displaymath}
  Choosing now the re-parametrization $\sfs_\eps$ corresponding to
  the choice 
  \begin{equation}
    \label{eq:DS:33}
    f_\eps(s):=\frac{1}{\sqrt{\eps^2+A^s}},\quad
    L_{f_\eps}:=\int_0^1 \sqrt{\eps^2+A^s}\,\d s,\quad
    L_{f_\eps}^2\le \eps^2+
    \int_0^1 A^s\,\d s,    
  \end{equation}
  we get
  \begin{align*}
    W^2(\bar\mu^{r_0},\bar\mu^{r_1})&\leq|r_1-r_0|\int_{r_0}^{r_1}
    \bar A^r\,\d r=
    |r_1-r_0| L_{f_\eps}^2 \int_{r_0}^{r_1} A^{\sfs(r)}
    f^2_\eps(\sfs(r))\,\d r
    \le (r_1-r_0)^2 L_{f_\eps}^2,
  \end{align*}
  which yields \eqref{eq:DS:29}.
 \end{proof}


\end{document}